\documentclass[12pt]{amsart}
\usepackage{amssymb,latexsym}
\usepackage{pdfsync}
\usepackage{color}

\usepackage[hidelinks]{hyperref}
\usepackage{esint}

\newdimen\AAdi%
\newbox\AAbo%
\def\AAk#1#2{\s_etbox\AAbo=\hbox{#2}\AAdi=\wd\AAbo\kern#1\AAdi{}}%
\def\AAr#1#2#3{\s_etbox\AAbo=\hbox{#2}\AAdi=\ht\AAbo\raise#1\AAdi\hbox{#3}}%
\font\tenmsb=msbm10 at 12pt
\font\sevenmsb=msbm7 at 8pt
\font\fivemsb=msbm5 at 6pt
\newfam\msbfam
\textfont\msbfam=\tenmsb
\scriptfont\msbfam=\sevenmsb
\scriptscriptfont\msbfam=\fivemsb
\def\Bbb#1{{\tenmsb\fam\msbfam#1}}

\newcommand{\beq}{\begin{equation}}
\newcommand{\eeq}{\end{equation}}
\newcommand{\beqr}{\begin{eqnarray}}
\newcommand{\eeqr}{\end{eqnarray}}
\newcommand{\ba}{\begin{array}}
\newcommand{\ea}{\end{array}}

\begin{document}

\newtheorem{theorem}{Theorem}[section]
\newtheorem{lemma}{Lemma}[section]
\newtheorem{corollary}{Corollary}[section]
\newtheorem{remark}{Remark}
\newtheorem{proposition}{Proposition}[section]
\newtheorem{definition}{Definition}
\newtheorem{eg}{Example}
\newtheorem*{claim}{Claim}
\newcommand{\noi}{\noindent}
\newcommand{\dis}{\displaystyle}
\newcommand{\mint}{-\!\!\!\!\!\!\int}
\numberwithin{equation}{section}

\def \bx{\hspace{2.5mm}\rule{2.5mm}{2.5mm}}
\def \vs{\vspace*{0.2cm}}
\def\hs{\hspace*{0.6cm}}
\def \ds{\displaystyle}
\def \p{\partial}
\def \O{\Omega}
\def \o{\omega}
\def \b{\beta}
\def \m{\mu}
\def \l{\lambda}
\def\L{\Lambda}
\def \ul{u_\lambda}
\def \D{\Delta}
\def \d{\delta}
\def \k{\kappa}
\def \s{\sigma}
\def \e{\varepsilon}
\def \a{\alpha}
\def \sm{\sigma}
\def \tf{\tilde{f}}
\def\cqfd{%
\mbox{ }%
\nolinebreak%
\hfill%
\rule{2mm} {2mm}%
\medbreak%
\par%
}
\def \pr {\noindent {\it Proof.} }
\def \rmk {\noindent {\it Remark} }
\def \esp {\hspace{4mm}}
\def \dsp {\hspace{2mm}}
\def \ssp {\hspace{1mm}}

\def\la{\langle}\def\ra{\rangle}

\def \u{u_+^{p^*}}
\def \ui{(u_+)^{p^*+1}}
\def \ul{(u^k)_+^{p^*}}
\def \energy{\int_{\R^n}\u }
\def \sk{\s_k}
\def \mo{\mu_k}
\def\cal{\mathcal}
\def \I{{\cal I}}
\def \J{{\cal J}}
\def \K{{\cal K}}
\def \OM{\overline{M}}

\def\n{\nabla}
\newcommand{\Ric}{\operatorname{Ric}}
\newcommand{\divg}{\operatorname{div}}

\def\fk{{{\cal F}}_k}
\def\M1{{{\cal M}}_1}
\def\Fk{{\cal F}_k}
\def\Fl{{\cal F}_l}
\def\FF{\cal F}
\def\Gk{{\Gamma_k^+}}
\def\n{\nabla}
\def\uuu{{\n ^2 u+du\otimes du-\frac {|\n u|^2} 2 g_0+S_{g_0}}}
\def\uuug{{\n ^2 u+du\otimes du-\frac {|\n u|^2} 2 g+S_{g}}}
\def\sku{\sk\left(\uuu\right)}
\def\qed{\cqfd}
\def\vvv{{\frac{\n ^2 v} v -\frac {|\n v|^2} {2v^2} g_0+S_{g_0}}}
\def\vvs{{\frac{\n ^2 \tilde v} {\tilde v}
 -\frac {|\n \tilde v|^2} {2\tilde v^2} g_{S^n}+S_{g_{S^n}}}}
\def\skv{\sk\left(\vvv\right)}
\def\tr{\hbox{tr}}
\def\pO{\partial \Omega}
\def\dist{\hbox{dist}}
\def\RR{\Bbb R}\def\R{\Bbb R}
\def\C{\Bbb C}
\def\B{\Bbb B}
\def\N{\Bbb N}
\def\Q{\Bbb Q}
\def\Z{\Bbb Z}
\def\PP{\Bbb P}
\def\EE{\Bbb E}
\def\F{\Bbb F}
\def\G{\Bbb G}
\def\H{\Bbb H}
\def\SS{\Bbb S}\def\S{\Bbb S}

\def\div{\hbox{div}\,}

\def\lcf{{locally conformally flat} }

\def\circledwedge{\setbox0=\hbox{$\bigcirc$}\relax \mathbin {\hbox
to0pt{\raise.5pt\hbox to\wd0{\hfil $\wedge$\hfil}\hss}\box0 }}

\def\sss{\frac{\s_2}{\s_1}}

\date{\today}

\title[Hamilton Gradient Estimates for PME and FDE]{Hamilton-Type Gradient Estimates for Porous Medium and Fast Diffusion Equations on Riemannian Manifolds for all exponents}

\author{}

 \author[Jun Sun]{Jun Sun } 
\address{School of Mathematics and Statistics\\ Wuhan University\\Wuhan 430072,
China
 }
 \email{sunjun@whu.edu.cn}

\author[Jiaming Yang]{Jiaming Yang}
\address{School of Mathematics and Statistics\\ Wuhan University\\Wuhan 430072,
China
 }
 \email{2020302142030@whu.edu.cn}

\begin{abstract}
We establish local Hamilton-type gradient estimates for positive
$C^{2,1}$ solutions of $u_t=\Delta u^m$ on complete Riemannian manifolds
whose Ricci curvature is bounded from below. For every fixed $m>1$ and
every fixed $0<m<1$, there exist $\beta=\beta(m,n)>0$ and $C=C(m,n)>0$
such that a solution  $0<u\leqslant A$ in
$B_{2R}(x_0)\times(t_0-T,t_0]$ satisfies the following local gradient estimate
\begin{equation*}
\sup_{B_R(x_0)\times(t_0-T/2,t_0]}|\nabla u^\beta|
\leqslant C A^\beta
\left(\frac1R+\sqrt{k}+\frac{A^{(1-m)/2}}{\sqrt T}\right),
\end{equation*}
where $\mathrm{Ric}_M\geqslant-k$. The proof uses an intrinsic quantitative
alternative to locate, around each prescribed positive point, a cylinder
on which the solution has a controlled upper-to-lower ratio. A local
gradient estimate on this cylinder is combined with a stopping argument. As a
consequence, every uniformly bounded positive ancient solution on a
connected complete manifold with nonnegative Ricci curvature is constant.


\end{abstract}
\maketitle
\section{Introduction}

\allowdisplaybreaks

In this paper we study local gradient estimates for the homogeneous nonlinear  equation
\begin{equation}\label{equ;1}
\frac{\partial u}{\partial t}=\Delta u^m
\quad\text{in}\quad
B_{2R}(x_0)\times(t_0-T,t_0],
\end{equation}
where $(M,g)$ is a complete Riemannian manifold of dimension $n\geqslant2$,
$B_{2R}(x_0)$ is a geodesic ball, and $T>0$. Throughout the paper,
$0<m<1$ or $m>1$. The equation is called the porous medium
equation (PME) when $m>1$ and the fast diffusion equation (FDE) when
$0<m<1$. In divergence form, its diffusion coefficient is $mu^{m-1}$.
Consequently, the equation degenerates near $u=0$ in the PME case and
becomes singular there in the FDE case. These different behaviors are
central to the regularity theory and to the construction of gradient
estimates; see, for example, \cite{LNVV,Liao}.

Gradient estimates connect local regularity with Harnack inequalities and
Liouville theorems. For the linear heat equation, the differential Harnack
estimates of Li and Yau \cite{LY} control a combination of spatial and
time derivatives. A complementary estimate of Hamilton
\cite{Hamilton} states that, on a closed Riemannian
manifold with $\mathrm{Ric}_M\geqslant-k$, a positive heat solution bounded
above by $A$ satisfies
\begin{equation*}
\frac{|\nabla u|^2}{u^2}
\leqslant\left(\frac1t+2k\right)\log\frac Au,
\qquad t>0.
\end{equation*}
In 2006, Souplet and Zhang \cite[Theorem~1.1]{SZ} established a localized
Hamilton-type estimate on complete noncompact manifolds. These results
provide two related, but distinct, models for nonlinear diffusion:
estimates involving a time derivative and estimates of a purely spatial
gradient.

For PME and FDE, the classical Aronson--B\'enilan inequality is an
important starting point. As recalled in \cite{LNVV},
for positive solutions on $\mathbb R^n$ it takes the form
\begin{equation*}
\Delta\left(\frac{m}{m-1}u^{m-1}\right)
\geqslant-\frac{\kappa}{t},
\qquad
\kappa=\frac{n}{n(m-1)+2},
\qquad m>1-\frac2n,\quad m\neq1.
\end{equation*}
Lu, Ni, V\'azquez and Villani \cite{LNVV} developed local
Aronson--B\'enilan and Li--Yau-type estimates on manifolds with a lower
Ricci curvature bound. Their Theorem~3.3 applies to every $m>1$, while
Theorem~4.1 treats $1-2/n<m<1$. A representative consequence of
Theorem~3.3 (1) is the following: if $\mathrm{Ric}_M\geqslant0$ on $B_R$,
$u$ is a positive smooth PME solution on $B_R\times[0,T]$, and
\begin{equation*}
v=\frac{m}{m-1}u^{m-1},\qquad
M_v=\sup_{B_R\times[0,T]}v<\infty,
\end{equation*}
then, for every $\alpha>1$,
\begin{equation*}
\frac{|\nabla v|^2}{v}-\alpha\frac{v_t}{v}
\leqslant C(n,m,\alpha)
\left(\frac1t+\frac{M_v}{R^2}\right)
\quad\text{on }B_{R/2}\times(0,T].
\end{equation*}
In particular, the full PME exponent range is already present in this
mixed space--time estimate. It does not by itself give an upper bound
for a purely spatial gradient, since the sign of $v_t$ is not prescribed.

Hamilton-type estimates for \eqref{equ;1} have also been obtained
by working directly with weighted pressure gradients. For FDE, Zhu
\cite{ZhuFDE} considered $1-\frac2n<m<1$ and the positive
inverse pressure
\begin{equation*}
\widetilde v=\frac{m}{1-m}u^{m-1}.
\end{equation*}
His local estimate controls $|\nabla\widetilde v|/\sqrt{\widetilde v}$
under an upper bound for $\widetilde v$. Since $m-1<0$, such a pressure
bound corresponds to a positive lower bound for $u$. For PME, Zhu
\cite{ZhuPME} obtained a local estimate for
$v^{(2-m)/(4(m-1))}|\nabla v|$ in the range
$1<m<1+\frac{1}{1+\sqrt{2n}}$. Huang, Xu and Zeng
\cite{HXZ} subsequently obtained an estimate for
$v^{1/(2(m-1))}|\nabla v|$ in the larger range
$1<m<1+\frac{1}{\sqrt{n-1}}$. Huang and Ma \cite{HM} established
further pressure-gradient estimates, including additional
 dimension-dependent FDE ranges. The quantities controlled in these
results differ, so their exponent intervals should not be compared
without also keeping track of the pressure weights and boundedness
assumptions.

A useful comparison in the FDE case is provided by Xu
\cite[Theorem~1.2]{Xu}. Let $u$ be a positive solution with $u\leqslant A$
on $B_R(x_0)\times(t_0-T,t_0]$ in a complete manifold satisfying
$\mathrm{Ric}_M\geqslant-k$. If
$1-4/(n+3)<m<1$, then
\begin{equation*}
\sup_{B_{R/2}(x_0)\times(t_0-T/2,t_0]}
\frac{|\nabla u|}{u}
\leqslant C(n,m)
\left(\frac1R+\sqrt{k}+\frac{A^{(1-m)/2}}{\sqrt T}\right).
\end{equation*}
This estimate requires no prescribed positive lower bound for $u$ and
controls the logarithmic gradient, but only in the stated exponent
range. Xu's Theorem~1.6 also gives a PME pressure-gradient estimate for
arbitrary $m>1$ under a quantitative pinching condition on the range of
$u$. Thus, control of the relative oscillation is already significant
in nonlinear gradient estimates.

More recently, Huang and Shen \cite{HS} used Moser iteration to obtain
local Li--Yau-type estimates for positive weak solutions. Their
Theorem~1.1 covers all $m>1$, and Theorem~1.8 covers
$1-2/n<m<1$. The constants in these local estimates depend on both the
maximum and the minimum of the relevant pressure on the working
cylinder. Their work demonstrates the usefulness of integral methods
for gradient estimates, while leaving a different issue from
the one addressed here: obtaining a purely spatial estimate, with no
quantitative lower bound for $u$, throughout both exponent ranges.

Our aim is to control $|\nabla u^\beta|$ for a suitable positive exponent
$\beta=\beta(m,n)$, for every fixed $m>1$ and every fixed $0<m<1$.
The estimate uses only the upper bound for $u$, the size of the
space--time cylinder, and the lower Ricci curvature bound. The exponent
$\beta$ is selected in the proof rather than prescribed as the pressure
exponent. This distinction is essential: extending the range of $m$
comes at the cost of estimating a sufficiently high positive power.
We do not claim that the resulting estimate improves the pressure or
logarithmic gradient bounds on the subranges where those stronger
quantities can already be controlled. Neither the optimal value of
$\beta$ nor uniformity of the constants as $m$ approaches an endpoint
is asserted.

The proof begins with a quantitative alternative on intrinsic cylinders
whose time length is proportional to $U^{1-m}r^2$. If the low-level set
has small space--time measure, a De Giorgi iteration produces a positive
lower bound on a smaller cylinder. Otherwise, comparison with the
solution of a uniformly parabolic linear equation, obtained from a
clipped coefficient, reduces the upper bound by a fixed factor. The
Sobolev inequality, local H\"older estimate and Dirichlet heat-kernel
lower bound needed here are supplied by Saloff-Coste
\cite[Theorem~3.1, Corollary~5.5 and Theorem~6.1]{S-C}. 

A closely related intrinsic De Giorgi strategy appears in Liao's work \cite{Liao}
on porous medium systems in $\mathbb{R^{N}}$ . His Theorem~1.1 establishes local H\"older regularity of bounded weak solutions for both $m>1$ and
$0<m<1$. In particular, Sections~3--5 distinguish between reduction
near zero and a regime in which the modulus of the solution is bounded
away from zero; Sections~5.1.2 and~5.2.2 then use normalized
nondegenerate equations. Our argument combines the alternative
with linear heat-kernel comparison and a stopping procedure centered
at each prescribed positive point.

More precisely, writing $u_*=u(x,t)>0$, the successive upper levels are
$U_j=\sigma^jA$ and the spatial scales are
$\rho_j=\sigma^{\beta j}\rho$. The process stops at a finite index $J$, with
$U_J/4\leqslant u\leqslant U_J$ on the relevant cylinder. Further
H\"older localization then yields a pressure-gradient estimate with
constants depending only on $m$ and $n$. The relation
\begin{equation*}
\rho_J\geqslant\rho\left(\frac{u_*}{A}\right)^\beta
\end{equation*}
compensates for the radius loss in the estimate of $\nabla u^\beta$.
For FDE, we use $(k^m-u^m)_+$ in the low-level energy estimate and place
the constant in the pressure weight below the local pressure values.

We use $C^{2,1}_{\mathrm{loc}}$ to express that the spatial derivatives through order two and the first time derivative are continuous on compact interior spatial
subcylinders, with their past-side traces at the upper time face. Our main results are the following local gradient estimates.

\begin{theorem}\label{thm;main}
Let $m>1$ and $n\geqslant 2$, and let $(M,g)$ be an $n$-dimensional complete
Riemannian manifold satisfying $\mathrm{Ric}_{M}\geqslant-k$ 
for some $k\geqslant0$.  Suppose that
\begin{equation*}
u\in C^{2,1}_{\mathrm{loc}}\bigl(B_{2R}(x_0)\times(t_0-T,t_0]\bigr)
\end{equation*}
is positive, satisfies (\ref{equ;1}) and
\begin{equation*}
0<u\leqslant A
\quad\text{in}\quad
B_{2R}(x_{0})\times(t_{0}-T,t_{0}].
\end{equation*}
Then there exist constants $\beta=\beta(m,n)>0$ and $C=C(m,n)>0$ such
that
\begin{equation}\label{equ;2}
\sup_{B_{R}(x_{0})\times(t_{0}-T/2,t_{0}]}
\left|\nabla u^{\beta}\right|
\leqslant
C(m,n)A^{\beta}
\left(
\frac{A^{-(m-1)/2}}{\sqrt{T}}
+\frac{1+\sqrt{k}R}{R}
\right).
\end{equation}
\end{theorem}

We can also prove the corresponding theorem for the fast diffusion equation.
\begin{theorem}\label{thm;main2}
Let $0<m<1$ and $n\geqslant 2$, and let $(M,g)$ be an $n$-dimensional complete
Riemannian manifold satisfying $\mathrm{Ric}_{M}\geqslant-k$ 
for some $k\geqslant0$.  Suppose that
\begin{equation*}
u\in C^{2,1}_{\mathrm{loc}}\bigl(B_{2R}(x_0)\times(t_0-T,t_0]\bigr)
\end{equation*}
is positive, satisfies (\ref{equ;1}) and 
\begin{equation*}
0<u\leqslant A
\quad\text{in}\quad
B_{2R}(x_{0})\times(t_{0}-T,t_{0}].
\end{equation*}
Then there exist constants $\beta=\beta(m,n)>0$ and $C=C(m,n)>0$ such
that
\begin{equation}\label{equ;71}
\sup_{B_{R}(x_{0})\times(t_{0}-T/2,t_{0}]}
\left|\nabla u^{\beta}\right|
\leqslant
C(m,n)A^{\beta}
\left(
\frac{A^{-(m-1)/2}}{\sqrt{T}}
+\frac{1+\sqrt{k}R}{R}
\right).
\end{equation}
\end{theorem}


\begin{corollary}\label{cor;Liouville}
Let $(M,g)$ be a connected complete Riemannian manifold of dimension $n\geqslant2$
with $\mathrm{Ric}_M\geqslant0$, and let $m>1$ or $0<m<1$.
Suppose that $u\in C^{2,1}_{\mathrm{loc}}(M\times(-\infty,t_{0}])$ satisfies
\begin{equation*}
       u_t=\Delta u^m,\qquad 0<u(x,t)\leqslant A<\infty
 \quad\text{on }M\times(-\infty,t_{0}].
\end{equation*}
Then $u$ is constant in both space and time.
\end{corollary}
\begin{remark}
Global boundedness in Corollary~1.1 can be replaced by a suitable
space--time growth condition. Let $m>0$, $m\neq1$, and let
$\beta=\beta(m,n)\geqslant m$ be the exponent selected in the proof. A sufficient condition is
\begin{equation*}
 u(x,t)=o\left(
 d(x,x_0)^{1/\beta}
 +(t_0-t)^{1/(2\beta-m+1)}
 \right)
\end{equation*}
uniformly as $d(x,x_0)+\sqrt{t_0-t}\to\infty$. The exponent $\beta$ is selected
sufficiently large in the proof and is not optimized. Consequently, a wider range of exponents does not by itself give sharper gradient estimates or weaker growth assumptions. We therefore do not claim that our results cover all previous gradient estimates or Liouville theorems.
\end{remark}

In Euclidean space, bounded ancient solutions of the corresponding porous
medium systems are already covered by Liao's Liouville theorem
\cite[Corollary~1.1]{Liao}. Corollary~\ref{cor;Liouville} concerns the
scalar equation on complete manifolds with nonnegative Ricci curvature
and follows from the local estimates above. 

The paper is organized as follows. Section~\ref{sec;alternative}
establishes the quantitative alternative for PME and FDE.
Section~\ref{sec;local-gradient} proves the local gradient estimates under
controlled two-sided bounds. Section~4 combines these ingredients with
the stopping argument to prove Theorems~\ref{thm;main}
and~\ref{thm;main2}.

\section{A Quantitative Alternative}\label{sec;alternative}

We first record the two analytic facts used in the iteration argument.  To
include the two-dimensional case, set
\begin{equation*}
n'=\begin{cases}
n, & n\geqslant3,\\
4, & n=2.
\end{cases}
\end{equation*}

\begin{lemma}[Saloff-Coste, \cite{S-C}]\label{lemma;Sobolev}
Let $(M,g)$ be a Riemannian manifold of dimension $n\geqslant2$ with
$\mathrm{Ric}_{M}\geqslant-k$ for some $k\geqslant0$.  There exists a
positive constant $C_{n'}$ depending only on $n'$ such that, for every
$\varphi\in C_{0}^{\infty}(B_{R})$, where $B_{R}\subset M$ is a geodesic
ball,
\begin{equation*}
e^{-C_{n'}(1+\sqrt{k}R)}|B_{R}|^{2/n'}R^{-2}
\left\|\varphi\right\|_{L^{\frac{2n'}{n'-2}}(B_{R})}^{2}
\leqslant
\int_{B_{R}}|\nabla\varphi|^{2}
+R^{-2}\int_{B_{R}}\varphi^{2}.
\end{equation*}
\end{lemma}

We shall also use the following elementary parabolic interpolation
consequence.

\begin{lemma}\label{lemma;parabolic-interpolation}
Let $(M,g)$ be a complete Riemannian manifold of dimension $n$, and let
\begin{equation*}
Q_{R,T}=B_{R}(x_{0})\times(t_{0}-T,t_{0}]
\subset M\times\mathbb{R}.
\end{equation*}
Suppose that $\varphi=\varphi(x,t)$ satisfies
\begin{equation}\label{equ;72}
X\sup_{t_{0}-\tau<t\leqslant t_{0}}
\int_{B_{\rho}}\varphi^{2}(x,t)
+Y\int_{t_{0}-\tau}^{t_{0}}
\left(
\int_{B_{\rho}}|\varphi|^{\frac{2n'}{n'-2}}
\right)^{\frac{n'-2}{n'}}\,dt
\leqslant S,
\end{equation}
where $B_{\rho}=B_{\rho}(x_{0})$ and $X,Y,S,\rho,\tau$ are constants
with $0<\rho\leqslant R$ and $0<\tau\leqslant T$.  Then
\begin{equation*}
\int_{Q_{\rho,\tau}}
|\varphi|^{\frac{2(n'+2)}{n'}}
\leqslant
\frac{S^{(n'+2)/n'}}{X^{2/n'}Y},
\end{equation*}
where
\begin{equation*}
Q_{\rho,\tau}
=B_{\rho}(x_{0})\times(t_{0}-\tau,t_{0}]
\subset Q_{R,T}.
\end{equation*}
Equivalently,
\begin{equation*}
\left\|\varphi\right\|_{L^{\frac{2(n'+2)}{n'}}(Q_{\rho,\tau})}^{2}
\leqslant
X^{-\frac{2}{n'+2}}Y^{-\frac{n'}{n'+2}}S.
\end{equation*}
\end{lemma}
\begin{proof}
      We assume that $n\geqslant 3$. The following argument is also vaild as long as $n$ is replaced by $n'=4$ when $n=2$.
    
    For $t_{0}-\tau<t\leqslant t_{0}$, we use H\"{o}lder's inequality to deduce that
    \begin{equation*}
    \begin{aligned}
    \int_{B_{\rho }}|\varphi| ^{\frac{2(n+2)}{n}}&=\int_{B_{\rho }}|\varphi |^{2}\cdot |\varphi| ^{\frac{4}{n}}\leqslant \left( \int_{B_{\rho }}|\varphi| ^{\frac{2n}{n-2}} \right) ^{\frac{n-2}{n}}\left( \int_{B_{\rho }}\varphi ^{2} \right) ^{\frac{2}{n}}\\
    &\leqslant \sup_{t_{0}-\tau<t\leqslant t_{0}}\left( \int_{B_{\rho }}\varphi ^{2} \right) ^{\frac{2}{n}}\left( \int_{B_{\rho }}|\varphi| ^{\frac{2n}{n-2}} \right) ^{\frac{n-2}{n}}. 
    \end{aligned}
    \end{equation*}
    Integrating over $t_{0}-\tau <t\leqslant t_{0}$ and applying the inequality in the assumption (\ref{equ;72}) complete the proof.
\end{proof}

For completeness, we next state the discrete iteration lemma that will be
used below.

\begin{lemma}[Discrete iteration lemma]\label{lemma;discrete-iteration}
Let $\alpha>0$, $C>0$, and $B>1$.  Suppose that a nonnegative sequence
$\{y_{j}\}_{j\geqslant0}$ satisfies
\begin{equation}\label{equ;3}
y_{j+1}\leqslant CB^{j}y_{j}^{1+\alpha},
\qquad j=0,1,2,\ldots.
\end{equation}
Then
\begin{equation}\label{equ;4}
y_{j}\leqslant
y_{0}B^{-j/\alpha}
\left(C^{1/\alpha}B^{1/\alpha^{2}}y_{0}\right)^{(1+\alpha)^{j}-1}.
\end{equation}
In particular, if
\begin{equation*}
y_{0}<C^{-1/\alpha}B^{-1/\alpha^{2}},
\end{equation*}
then $y_{j}\to0$ as $j\to\infty$.
\end{lemma}

\begin{proof}
Set
\begin{equation*}
z_{j}=C^{1/\alpha}B^{j/\alpha+1/\alpha^{2}}y_{j}.
\end{equation*}
It follows from (\ref{equ;3}) that
\begin{equation*}
z_{j+1}\leqslant z_{j}^{1+\alpha}.
\end{equation*}
Consequently, $z_{j}\leqslant z_{0}^{(1+\alpha)^{j}}$, which is exactly
(\ref{equ;4}).
\end{proof}

We now establish the alternative needed to obtain a sufficiently small
cylinder on which the oscillation of $u$ is controlled.

\begin{lemma}[Alternative lemma]\label{lemma;alternative}
Let $m>1$ and $(M,g)$ be an $n$-dimensional complete
Riemannian manifold satisfying $\mathrm{Ric}_{M}\geqslant-k$. Suppose $u$ is a positive solution of (\ref{equ;1}), and  $(x,t)\in B_{R}(x_{0})\times(t_{0}-T/2,t_{0}]$. Then there exist constants
$\nu=\nu(m,n)$ and $\sigma=\sigma(m,n)$ satisfying
\begin{equation*}
0<\nu\leqslant\frac14,
\qquad
\frac34\leqslant\sigma<1,
\end{equation*}
with the following property.  Suppose that $\sqrt{k}r\leqslant1/8$ and
\begin{equation*}
0<u\leqslant U
\quad\text{in}\quad
B_{2r}(x)\times(t-4U^{1-m}r^{2},t].
\end{equation*}
Then one of the following two alternatives occurs:
\begin{align*}
\mathrm{(i)}\quad &u\geqslant\frac14U
&&\text{a.e. in }B_{r/2}(x)\times
\left(t-\frac14U^{1-m}r^{2},t\right],\\
\mathrm{(ii)}\quad &u\leqslant\sigma U
&&\text{a.e. in }B_{r/2}(x)\times
\left(t-\frac{\nu}{4}U^{1-m}r^{2},t\right].
\end{align*}
\end{lemma}

\begin{proof}
For simplicity, write $B_{r}=B_{r}(x)$ and set
\begin{equation*}
Q=B_{r}\times(t-U^{1-m}r^{2},t],
\qquad
E=Q\cap\{u<U/2\}.
\end{equation*}
We first consider the case $|E|\leqslant \nu |Q|$, that is, 
\begin{equation}\label{equ;5}
|E|\leqslant\nu U^{1-m}r^{2}|B_{r}|,
\end{equation}
where $\nu$ will be chosen below.

For $j=0,1,2,\ldots$, define
\begin{gather*}
k_{j}=\frac{U}{4}+\frac{U}{2^{j+2}},
\qquad
r_{j}=\frac{r}{2}+\frac{r}{2^{j+1}},
\qquad
f_{j}=(k_{j}-u)_{+},\\
Q_{j}=B_{r_{j}}\times(t-U^{1-m}r_{j}^{2},t],
\qquad
E_{j}=Q_{j}\cap\{u<k_{j}\},
\qquad
s_{j}=t-U^{1-m}r_{j}^{2}.
\end{gather*}
Choose a space-time cut-off function $\zeta_{j}$ such that
$\zeta_{j}\equiv1$ on $Q_{j+1}$,
$\operatorname{Supp}\zeta_{j}\subset Q_{j}$, and $0\leqslant\zeta_{j}\leqslant1$.
It may be chosen so that
\begin{equation}\label{equ;6}
|\nabla\zeta_{j}|\leqslant\frac{C2^{j}}{r_{j}},
\qquad
\left|\frac{\partial\zeta_{j}}{\partial t}\right|
\leqslant\frac{C4^{j}U^{m-1}}{r_{j}^{2}}.
\end{equation}

Write (\ref{equ;1}) in divergence form as $u_{t}=\operatorname{div}(mu^{m-1}\nabla u)$. Testing this equation with $f_{j}\zeta_{j}^{2}$ on
\begin{equation*}
Q_{j}^{\tau}=B_{r_{j}}\times(s_{j},\tau],
\qquad s_{j}\leqslant\tau\leqslant t,
\end{equation*}
we obtain 
\begin{align*}
&\frac12\int_{B_{r_{j}}}f_{j}^{2}\zeta_{j}^{2}(\cdot,\tau)
-\int_{Q_{j}^{\tau}}f_{j}^{2}\zeta_{j}(\zeta_{j})_{t}\\
&\quad=-m\int_{Q_{j}^{\tau}}u^{m-1}|\nabla f_{j}|^{2}\zeta_{j}^{2}
-2m\int_{Q_{j}^{\tau}}u^{m-1}f_{j}\zeta_{j}
\langle\nabla f_{j},\nabla\zeta_{j}\rangle.
\end{align*}
By the Cauchy-Schwarz  inequality,
\begin{align*}
&-2m\int_{Q_{j}^{\tau}}u^{m-1}f_{j}\zeta_{j}
\langle\nabla f_{j},\nabla\zeta_{j}\rangle\\
&\quad\leqslant
\frac{m}{2}\int_{Q_{j}^{\tau}}u^{m-1}|\nabla f_{j}|^{2}\zeta_{j}^{2}
+2m\int_{Q_{j}^{\tau}}u^{m-1}f_{j}^{2}|\nabla\zeta_{j}|^{2}.
\end{align*}
Using (\ref{equ;6}), $f_{j}\leqslant U/2$, and $u\leqslant U$, we obtain
\begin{equation}\label{equ;7}
\frac12\int_{B_{r_{j}}}f_{j}^{2}\zeta_{j}^{2}(\cdot,\tau)
+\frac{m}{2}\int_{Q_{j}^{\tau}}u^{m-1}|\nabla f_{j}|^{2}\zeta_{j}^{2}
\leqslant
\frac{C4^{j}U^{m+1}}{r^{2}}|E_{j}|.
\end{equation}
Define $h_{j}=\left(k_{j}-\max\{u,U/8\}\right)_{+}$.
Then $0\leqslant h_{j}\leqslant f_{j}$, $\nabla h_{j}=0$ on
$\{u\leqslant U/8\}$, and $|\nabla h_{j}|=|\nabla f_{j}|$ on
$\{u>U/8\}$.  Hence
\begin{equation*}
|\nabla h_{j}|^{2}
\leqslant\left(\frac{8u}{U}\right)^{m-1}|\nabla f_{j}|^{2},
\end{equation*}
which implies
\begin{equation*}
u^{m-1}|\nabla f_{j}|^{2}
\geqslant\left(\frac{U}{8}\right)^{m-1}|\nabla h_{j}|^{2}
\end{equation*}
almost everywhere. Since
\begin{equation*}
|\nabla(\zeta_{j}h_{j})|^{2}
\leqslant2\zeta_{j}^{2}|\nabla h_{j}|^{2}
+2h_{j}^{2}|\nabla\zeta_{j}|^{2},
\end{equation*}
(\ref{equ;7}) yields
\begin{equation}\label{equ;8}
\begin{aligned}
&\frac12\int_{B_{r_{j}}}h_{j}^{2}\zeta_{j}^{2}(\cdot,\tau)
+\frac{m}{4}\left(\frac{U}{8}\right)^{m-1}
\int_{Q_{j}^{\tau}}|\nabla(\zeta_{j}h_{j})|^{2}\\
&\quad\leqslant
\frac{m}{2}\left(\frac{U}{8}\right)^{m-1}
\int_{Q_{j}^{\tau}}h_{j}^{2}|\nabla\zeta_{j}|^{2}
+\frac{C4^{j}U^{m+1}}{r^{2}}|E_{j}|\\
&\quad\leqslant
\frac{C4^{j}U^{m+1}}{r^{2}}|E_{j}|.
\end{aligned}
\end{equation}

Applying Lemma \ref{lemma;Sobolev} to $\zeta_{j}h_{j}$ in $B_{r}$, integrating in
time, and then using (\ref{equ;8}), we obtain
\begin{align*}
&\frac12\int_{B_{r_{j}}}h_{j}^{2}\zeta_{j}^{2}(\cdot,\tau)
+e^{-C_{n'}(1+\sqrt{k}r)}|B_{r}|^{2/n'}r^{-2}
\frac{m}{4}\left(\frac{U}{8}\right)^{m-1}\int_{s_{j}}^{\tau}
\left\|\zeta_{j}h_{j}\right\|_{L^{\frac{2n'}{n'-2}}(B_{r_{j}})}^{2}
\,dt\\
&\quad\leqslant
\frac{m}{4}\left(\frac{U}{8}\right)^{m-1}r^{-2}
\int_{Q_{j}^{\tau}}\zeta_{j}^{2}h_{j}^{2}
+\frac{C4^{j}U^{m+1}}{r^{2}}|E_{j}|\\
&\quad\leqslant
\frac{C4^{j}U^{m+1}}{r^{2}}|E_{j}|.
\end{align*}
Take the supremum on the left hand side of the above inequality with respect to $\tau$. Then Lemma \ref{lemma;parabolic-interpolation} gives
\begin{align*}
\int_{Q_{j}}(\zeta_{j}h_{j})^{\frac{2(n'+2)}{n'}}
&\leqslant
e^{C_{n'}(1+\sqrt{k}r)}|B_{r}|^{-2/n'}r^{2}
\left(\frac{m}{4}\right)^{-1}
\left(\frac{U}{8}\right)^{1-m}\\
&\qquad\times
\left(\frac{C4^{j}U^{m+1}}{r^{2}}|E_{j}|\right)^{\frac{n'+2}{n'}}.
\end{align*}
On $E_{j+1}\subset Q_{j}$, we have $\zeta_{j}=1$ and
\begin{equation*}
h_{j}\geqslant k_{j}-k_{j+1}=\frac{U}{2^{j+3}}.
\end{equation*}
Using also $\sqrt{k}r\leqslant1/8$, we infer
\begin{equation*}
|E_{j+1}|
\leqslant
\frac{C(m,n)4^{\frac{n'+2}{n'}(2j+3)}U^{\frac{2(m-1)}{n'}}}
{|B_{r}|^{2/n'}r^{4/n'}}
|E_{j}|^{1+2/n'}.
\end{equation*}
Equivalently,
\begin{equation*}
\frac{U^{m-1}}{|B_{r}|r^{2}}|E_{j+1}|
\leqslant
C(m,n)4^{\frac{2(n'+2)}{n'}j}
\left(
\frac{U^{m-1}}{|B_{r}|r^{2}}|E_{j}|
\right)^{1+2/n'}.
\end{equation*}
Choose
\begin{equation}\label{equ;9}
\nu=\min\left\{
\frac14,
\frac12 C(m,n)^{-n'/2}2^{-n'(n'+2)}
\right\}.
\end{equation}
By Lemma \ref{lemma;discrete-iteration} and (\ref{equ;5}),
$|E_{j}|\to0$.  Since $k_{j}\to  U/4$, alternative (i) follows.

It remains to consider the case in which (\ref{equ;5}) fails, namely
\begin{equation}\label{equ;10}
|E|=\int_{t-U^{1-m}r^{2}}^{t}|E_{\tau}|\,\,d\tau
>\nu U^{1-m}r^{2}|B_{r}|,
\qquad
E_{\tau}=\{y\in B_{r}:u(y,\tau)<U/2\}.
\end{equation}
Split the time interval into
\begin{equation}\label{equ;11}
I_{1}=\left(t-U^{1-m}r^{2},
t-\frac{\nu}{2}U^{1-m}r^{2}\right],
\qquad
I_{2}=\left(t-\frac{\nu}{2}U^{1-m}r^{2},t\right].
\end{equation}
Since
\begin{equation*}
\int_{I_{2}}|E_{\tau}|\,\,d\tau
\leqslant\frac{\nu}{2}U^{1-m}r^{2}|B_{r}|,
\end{equation*}
there exists $s\in I_{1}$ such that
\begin{equation}\label{equ;12}
E_{s}=\{y\in B_{r}:u(y,s)<U/2\},
\qquad
|E_{s}|\geqslant\frac{\nu}{2}|B_{r}|.
\end{equation}

On $B_{2r}(x)\times(t-4U^{1-m}r^{2},t]$, define
\begin{equation}\label{equ;13}
z=\Phi(u)=\min\{U-u,U/2\}
=\begin{cases}
U/2, & u<U/2,\\
U-u, & u\geqslant U/2.
\end{cases}
\end{equation}
In the sense of distributions,
\begin{align*}
z_{t}-\operatorname{div}(mu^{m-1}\nabla z)
&=\Phi'(u)u_{t}
-\operatorname{div}\bigl(mu^{m-1}\Phi'(u)\nabla u\bigr)\\
&=\Phi'(u)u_{t}-\operatorname{div}\bigl(\Phi'(u)\nabla u^{m}\bigr)\\
&=-mu^{m-1}\Phi''(u)|\nabla u|^{2}\geqslant0.
\end{align*}
Since $\nabla z=0$ when $u<U/2$, we may write
\begin{equation*}
mu^{m-1}\nabla z=d(y,\tau)\nabla z,
\qquad
d(y,\tau)=m\max\{u(y,\tau),U/2\}^{m-1},
\end{equation*}
where
\begin{equation}\label{equ;14}
m2^{1-m}U^{m-1}\leqslant d(y,\tau)\leqslant mU^{m-1}.
\end{equation}

For the time $s$ selected in (\ref{equ;12}), consider the initial-boundary
value problem
\begin{equation}\label{equ;15}
\begin{cases}
w_{\tau}=\operatorname{div}(d(y,\tau)\nabla w)
&\text{in }B_{3r/2}(x)\times(s,t],\\
w=0
&\text{on }\partial B_{3r/2}(x)\times(s,t],\\
w(\cdot,s)=\dfrac{U}{2}\mathbf{1}_{E_{s}}
&\text{in }B_{3r/2}(x).
\end{cases}
\end{equation}
If $K(\tau,y;s,\xi)$ denotes the corresponding Dirichlet heat kernel, then by the minimality of the Dirichlet heat kernel (see \cite[Corollary 10.5]{Li} for example), $K$ is the minimal fundamental solution in $B_{3r/2}(x)\times(s,t]$ (see \cite[Section 6]{S-C}). The solution  of (\ref{equ;15}) is given by 
\begin{equation*}
w(y,\tau)=\frac{U}{2}\int_{E_{s}}K(\tau,y;s,\xi)\,\,d\mu(\xi).
\end{equation*}
Since $z$ is a supersolution and
$z(\cdot,s)\geqslant(U/2)\mathbf{1}_{E_{s}}$, the comparison principle
implies
\begin{equation}\label{equ;16}
z(y,\tau)\geqslant w(y,\tau)
=\frac{U}{2}\int_{E_{s}}K(\tau,y;s,\xi)\,\,d\mu(\xi).
\end{equation}

On the other hand, rescale the time variable by
\begin{equation*}
\theta=U^{m-1}(\tau-s),
\qquad
\widehat{w}(y,\theta)=w(y,s+U^{1-m}\theta),
\end{equation*}
and define
\begin{equation*}
\widehat{d}(y,\theta)
=U^{1-m}d(y,s+U^{1-m}\theta).
\end{equation*}
Then $m2^{1-m}\leqslant\widehat d\leqslant m$, and
\begin{equation}\label{equ;17}
\begin{cases}
\widehat w_{\theta}=\operatorname{div}(\widehat d(y,\theta)\nabla\widehat w)
&\text{in }B_{3r/2}(x)\times(0,U^{m-1}(t-s)],\\
\widehat w=0
&\text{on }\partial B_{3r/2}(x)\times(0,U^{m-1}(t-s)],\\
\widehat w(\cdot,0)=\dfrac{U}{2}\mathbf{1}_{E_{s}}
&\text{in }B_{3r/2}(x).
\end{cases}
\end{equation}
If $\widehat K (\theta ,y;0,\xi)$ is the minimal fundamental solution of the rescaled problem, then
\begin{equation}\label{equ;18}
K(\tau,y;s,\xi)
=\widehat K\bigl(U^{m-1}(\tau-s),y;0,\xi\bigr).
\end{equation}

Fix any $y\in B_{r/2}(x)$ and $\tau\in\left(t-\frac{\nu}{4}U^{1-m}r^{2},t\right]$.
Since $s\in I_{1}$,
\begin{equation*}
\frac{\nu}{4}U^{1-m}r^{2}
\leqslant\tau-s\leqslant U^{1-m}r^{2}.
\end{equation*}
Denote $h=U^{m-1}(\tau-s)$.  Then
$0<\nu r^{2}/4\leqslant h\leqslant r^{2}<(5r/4)^{2}$.  Taking
$\delta=4/5$, $B=B_{5r/4}(x)$, and $\Omega=B_{3r/2}(x)$ in
Theorem 6.1 of \cite{S-C}, we obtain
\begin{equation}\label{equ;19}
\widehat K(h,y;0,\xi)
\geqslant
\frac{e^{-C_{1}(1+kh)}}
{|B_{\sqrt h}(y)|^{1/2}|B_{\sqrt h}(\xi)|^{1/2}}
\exp\left(-\frac{C_{2}d^{2}(y,\xi)}{h}\right),
\end{equation}
where $C_{1}$ and $C_{2}$ depend only on $m$ and $n$.
For $\xi\in E_{s}\subset B_{r}(x)$, we have $d(y,\xi)\leqslant3r/2$,
$kh\leqslant kr^{2}\leqslant1/64$, and
$B_{\sqrt h}(y),B_{\sqrt h}(\xi)\subset B_{2r}(x)$.  Thus
\begin{equation*}
K(\tau,y;s,\xi)
\geqslant
\frac{\exp\left[-\left(65C_{1}/64+9C_{2}/\nu\right)\right]}
{|B_{2r}|}.
\end{equation*}
By volume comparison theorem and 
\begin{equation*}
\omega_{n}r^{n}\leqslant V_{k}(r)
\leqslant\omega_{n}r^{n}e^{\sqrt{(n-1)k}\,r},
\qquad r>0,\quad k\geqslant0,
\end{equation*}
where $\omega_{n}$ is the volume of the unit ball in $\mathbb{R}^{n}$ and $V_{k}(r)$ denotes the volume of a ball of radius $r$ in the space form of constant sectional curvature $-k/(n-1)$, we have
\begin{equation}\label{equ;20}
|B_{2r}|
\leqslant|B_{r}|\frac{V_{k}(2r)}{V_{k}(r)}
\leqslant2^{n}e^{2\sqrt{(n-1)k}\,r}|B_{r}|
\leqslant2^{n}e^{\sqrt{n-1}/4}|B_{r}|.
\end{equation}
Let
\begin{equation*}
C_{K}(m,n,\nu)
=\max\left\{
2^{n}\exp\left(
\frac{\sqrt{n-1}}{4}+\frac{65C_{1}}{64}+\frac{9C_{2}}{\nu}
\right),1
\right\}\geqslant 1.
\end{equation*}
It follows from (\ref{equ;18}), (\ref{equ;19}) and (\ref{equ;20}) that
\begin{equation*}
K(\tau,y;s,\xi)\geqslant\frac{1}{C_{K}|B_{r}|}.
\end{equation*}
Substituting this inequality into (\ref{equ;16}) and using (\ref{equ;12}), 
\begin{equation*}
z(y,\tau)
\geqslant\frac{U}{2}\frac{|E_{s}|}{C_{K}|B_{r}|}
\geqslant\frac{\nu U}{4C_{K}}.
\end{equation*}
Since $z\leqslant U-u$, we conclude that
\begin{equation}\label{equ;21}
u\leqslant\left(1-\frac{\nu}{4C_{K}}\right)U
\quad\text{in}\quad
B_{r/2}(x)\times
\left(t-\frac{\nu}{4}U^{1-m}r^{2},t\right].
\end{equation}
Taking $\sigma=1-\nu/(4C_{K})$ proves alternative (ii).
\end{proof}

For the FDE case, we prove the following result:
\begin{lemma}[Alternative lemma]\label{lemma;alternative2}
Let $0<m<1$ and $(M,g)$ be an $n$-dimensional complete
Riemannian manifold satisfying $\mathrm{Ric}_{M}\geqslant-k$. Suppose $u$ is a positive solution of \eqref{equ;1}, and $(x,t)\in B_{R}(x_{0})\times(t_{0}-T/2,t_{0}]$. Then there exist constants $\nu=\nu(m,n)$ and $\sigma=\sigma(m,n)$ satisfying
\begin{equation*}
 0<\nu\leqslant\frac14,\qquad\frac34\leqslant\sigma<1,
\end{equation*}
with the following property. Suppose that $\sqrt{k}r\leqslant1/8$ and
\begin{equation*}
 0<u\leqslant U\quad\text{in }
 B_{2r}(x)\times(t-4U^{1-m}r^2,t].
\end{equation*}
Then at least one of the following alternatives holds:
\begin{align*}
 \textup{(i)}\quad&u\geqslant\frac U4
 &&\text{a.e. in }\ B_{r/2}(x)\times\left(t-\frac{1}{4}U^{1-m}r^2,t\right],\\
 \textup{(ii)}\quad&u\leqslant\sigma U
 &&\text{a.e. in }\ B_{r/2}(x)\times\left(t-\frac{1}{4}\nu U^{1-m}r^2,t\right].
\end{align*}
\end{lemma}

\begin{proof}
Write $B_r=B_r(x)$ and set
\begin{equation*}
 Q=B_r\times(t-U^{1-m}r^2,t],\qquad E=Q\cap\{u<U/2\}.
\end{equation*}
First suppose that
\begin{equation}\label{equ;35}
 |E|\leqslant\nu U^{1-m}r^2|B_r|.
\end{equation}
As in the proof of Lemma \ref{lemma;alternative}, define
\begin{gather*}
 k_j=\frac U4+\frac U{2^{j+2}},\qquad
 r_j=\frac r2+\frac r{2^{j+1}},\qquad f_j=(k_j-u)_+,\\
 s_j=t-U^{1-m}r_j^2,\qquad
 Q_j=B_{r_j}\times(s_j,t],\qquad E_j=Q_j\cap\{u<k_j\}.
 \end{gather*}
Choose $0\leqslant\zeta_j\leqslant1$, equal to one on $Q_{j+1}$, vanishing on the spatial boundary and initial time of $Q_j$, with
\begin{equation}\label{equ;36}
 |\nabla\zeta_j|\leqslant\frac{C2^j}{r},\qquad
 |(\zeta_j)_t|\leqslant\frac{C4^jU^{m-1}}{r^2}.
\end{equation}
Define
\begin{equation}\label{equ;37}
 z_j=(k_j^m-u^m)_+,\qquad
 H_k(u)=\begin{cases}
 \displaystyle\int_u^k(k^m-s^m)\,ds,&u<k,\\[2pt]
 0,&u\geqslant k.
 \end{cases}
\end{equation}
Testing $u_t=\Delta u^m$ with $-z_j\zeta_j^2$, and writing $Q_j^\tau=B_{r_j}\times(s_j,\tau]$, we obtain 
\begin{equation}\label{equ;38}
\begin{aligned}
 &\int_{B_{r_j}}H_{k_j}(u)\zeta_j^2(\cdot,\tau)\,d\mu
   +\int_{Q_j^\tau}|\nabla z_j|^2\zeta_j^2\,d\mu\,ds\\
 &\quad=2\int_{Q_j^\tau}H_{k_j}(u)\zeta_j(\zeta_j)_t\,d\mu\,ds
       -2\int_{Q_j^\tau}z_j\zeta_j
           \langle\nabla z_j,\nabla\zeta_j\rangle\,d\mu\,ds.
\end{aligned}
\end{equation}
Indeed, $H_k'(u)=-(k^m-u^m)_+$ and $\nabla z_j=-\nabla u^m$ on $\{u<k_j\}$.  For $0\leqslant s\leqslant k$ and $0<m<1$,
\begin{equation*}
 mk^{m-1}(k-s)\leqslant k^m-s^m\leqslant k^{m-1}(k-s).
\end{equation*}
The lower inequality follows from concavity, and the upper one from $(s/k)^m\geqslant s/k$. Integration yields
\begin{equation}\label{equ;39}
 \frac m2k^{m-1}(k-u)_+^2
 \leqslant H_k(u)\leqslant\frac12k^{m-1}(k-u)_+^2.
\end{equation}
Moreover,
\begin{equation*}
 H_{k_j}(u)\leqslant C_mU^{m+1}\mathbf1_{E_j},\qquad
 z_j^2\leqslant U^{2m}\mathbf1_{E_j}.
\end{equation*}
Using Young's inequality in \eqref{equ;38}, applying \eqref{equ;36}, and taking the supremum with respect to $\tau$, we obtain
\begin{equation}\label{equ;40}
\begin{aligned}
 &\sup_{s_j<\tau\leqslant t}\int_{B_{r_j}}H_{k_j}(u)\zeta_j^2(\cdot,\tau)\,d\mu
   +\frac12\int_{Q_j}|\nabla z_j|^2\zeta_j^2\,d\mu\,ds\\
 &\hspace{35mm}\leqslant \frac{C(m,n)4^jU^{2m}}{r^2}|E_j|.
\end{aligned}
\end{equation}
Since $k_j\leqslant U$ and $m-1<0$, on the non-zero set
\begin{equation*}
 H_{k_j}(u)\geqslant\frac m2U^{m-1}f_j^2,\qquad
 |\nabla f_j|=\frac{u^{1-m}}m|\nabla z_j|
 \leqslant\frac{U^{1-m}}m|\nabla z_j|.
\end{equation*}
Multiplying \eqref{equ;40} by $U^{1-m}/m$ and using
 $|\nabla(\zeta_jf_j)|^2 \leqslant2\zeta_j^2|\nabla f_j|^2+2f_j^2|\nabla\zeta_j|^2$, 
\begin{equation}\label{equ;41}
\begin{aligned}
 &\frac12\sup_{s_j<\tau\leqslant t}\int_{B_{r_j}}(\zeta_jf_j)^2(\cdot,\tau)\,d\mu+\frac m4U^{m-1}\int_{Q_j}
      |\nabla(\zeta_jf_j)|^2\,d\mu\,ds
 \leqslant\frac{C(m,n)4^jU^{m+1}}{r^2}|E_j|.
\end{aligned}
\end{equation}
At this point the Sobolev and discrete iteration steps are the same as in the proof of Lemma \ref{lemma;alternative}, with $f_j$ in place of $h_j$. For clarity, applying Lemmas~\ref{lemma;Sobolev}--\ref{lemma;parabolic-interpolation} to \eqref{equ;41}, with
\begin{equation*}
 X=\frac12,\qquad
 Y=\frac m4U^{m-1}e^{-C_{n'}(1+\sqrt{k}r)}|B_r|^{2/n'}r^{-2},
\end{equation*}
gives
\begin{equation}\label{equ;42}
\begin{aligned}
 \int_{Q_j}|\zeta_jf_j|^{2(n'+2)/n'}\,d\mu\,ds
 &\leqslant C(m,n)e^{C_{n'}(1+\sqrt{k}r)}|B_r|^{-2/n'}r^2U^{1-m}\\[-2pt]
 &\qquad\times
 \left(\frac{4^jU^{m+1}}{r^2}|E_j|\right)^{(n'+2)/n'}.
\end{aligned}
\end{equation}
On $E_{j+1}$, $\zeta_j=1$ and $f_j\geqslant k_j-k_{j+1}=U2^{-j-3}$. Since $\sqrt{k}r\leqslant1/8$, it follows that
\begin{equation}\label{equ;43}
 |E_{j+1}|\leqslant C_0\,2^{4j(1+2/n')}
    \bigl(U^{1-m}r^2|B_r|\bigr)^{-2/n'}|E_j|^{1+2/n'},
\end{equation}
where $C_0=C_0(m,n)\geqslant1$. Notice that the remaining power of $U$ is $2(m-1)/n'$, which has the opposite sign from the PME case but is removed by the same normalization.

Set
\begin{equation}\label{equ;44}
 \nu=\min\left\{\frac14,\frac12C_0^{-n'/2}2^{-n'(n'+2)}\right\},
 \qquad y_j=\frac{|E_j|}{U^{1-m}r^2|B_r|}.
\end{equation}
Then $y_0\leqslant\nu$ and \eqref{equ;43} is \eqref{equ;3} with $\alpha=2/n'$ and $B=2^{4(1+2/n')}$. Lemma~\ref{lemma;discrete-iteration} yields $|E_j|\to0$. The set $\{u<U/4\}$ in the cylinder of alternative (i) is contained in every $E_j$, so it has measure zero. This implies (i).

Suppose now that \eqref{equ;35} fails. The same time splitting and Fubini argument as in the  proof of Lemma \ref{lemma;alternative} produce
\begin{equation}\label{equ;45}
\begin{aligned}
 &s\in(t-U^{1-m}r^2,t-\nu U^{1-m}r^2/2),\\
 &E_s=\{y\in B_r:u(y,s)<U/2\},\qquad |E_s|\geqslant\frac\nu2|B_r|.
\end{aligned}
\end{equation}
Retain the truncation and the clipped coefficient
\begin{equation}\label{equ;46}
 z=\min\{U-u,U/2\},\qquad
 d(y,\tau)=m\max\{u(y,\tau),U/2\}^{m-1}.
\end{equation}
Then
\begin{equation*}
 z_\tau-\operatorname{div}(mu^{m-1}\nabla z)\geqslant0.
\end{equation*}
The coefficient bounds are now
\begin{equation}\label{equ;47}
 mU^{m-1}\leqslant d(y,\tau)\leqslant m2^{1-m}U^{m-1}.
\end{equation}

Let $w$ solve
\begin{equation}\label{equ;48}
\begin{cases}
w_{\tau}=\operatorname{div}(d(y,\tau)\nabla w)
&\text{in }B_{3r/2}(x)\times(s,t],\\
w=0
&\text{on }\partial B_{3r/2}(x)\times(s,t],\\
w(\cdot,s)=\dfrac{U}{2}\mathbf{1}_{E_{s}}
&\text{in }B_{3r/2}(x).
\end{cases}
\end{equation}
Follow the remaining steps in the proof of Lemma \ref{lemma;alternative} and thus alternative (ii) follows with $\sigma=\sigma(m,n)\in [3/4,1)$.
\end{proof}


\section{Local Gradient Estimates Under Two-Sided Bounds}\label{sec;local-gradient}

\begin{lemma}\label{lemma;local-gradient}
Let $u\in C^{2,1}_{\mathrm{loc}}$ be a positive  solution of (\ref{equ;1}) for $m>1$, and let
$(x,t)\in B_{R}(x_{0})\times(t_{0}-T/2,t_{0}]$. Suppose that $\mathrm{Ric}_{M}\geqslant-k$, $\sqrt{k}r\leqslant1/8$, and that for some $U>0$,
\begin{equation}\label{equ;22}
\frac{U}{4}\leqslant u\leqslant U
\quad\text{in}\quad
B_{r/2}(x)\times
\left(t-\frac14U^{1-m}r^{2},t\right].
\end{equation}
Then there exists a constant $C=C(m,n)$ such that
\begin{equation}\label{equ;23}
|\nabla u|(x,t)\leqslant C(m,n)\frac{u(x,t)}{r}.
\end{equation}
\end{lemma}

\begin{proof}
Set
\begin{equation*}
a(y,\tau)=mu^{m-1}(y,\tau).
\end{equation*}
Then $u$ solves the equation
\begin{equation*}
u_{\tau}=\operatorname{div}(a(y,\tau)\nabla u).
\end{equation*}
Rescale the time variable by
\begin{equation*}
\theta=U^{m-1}(\tau-t),
\qquad
\widehat u(y,\theta)=u(y,t+U^{1-m}\theta),
\end{equation*}
and put
\begin{equation*}
\widehat a(y,\theta)=U^{1-m}a(y,t+U^{1-m}\theta).
\end{equation*}
Then
\begin{equation}\label{equ;24}
\widehat u_{\theta}
=\operatorname{div}(\widehat a(y,\theta)\nabla\widehat u)
\quad\text{in}\quad
B_{r/2}(x)\times\left(-\frac{1}{4}r^{2},0\right].
\end{equation}
Moreover, by (\ref{equ;22}),
\begin{equation*}
\frac{m}{4^{m-1}}\leqslant\widehat a(y,\theta)\leqslant m.
\end{equation*}
By the local H\"older continuity in Corollary 5.5 of \cite{S-C}, there
exist constants $C_{H}=C_{H}(m,n)$ and
$\alpha_{H}=\alpha_{H}(m,n)\in(0,1)$ such that
\begin{equation*}
|\widehat u(y,\theta)-\widehat u(x,0)|
\leqslant
C_{H}\|\widehat u\|_{\infty}
\left(\frac{d(x,y)+\sqrt{-\theta}}{r}\right)^{\alpha_{H}}
\end{equation*}
in $B_{r/4}(x)\times(-r^{2}/16,0]$.  Returning to the original time
variable gives
\begin{equation}\label{equ;25}
|u(y,\tau)-u(x,t)|
\leqslant
C_{H}U
\left(
\frac{d(x,y)+\sqrt{U^{m-1}(t-\tau)}}{r}
\right)^{\alpha_{H}}
\end{equation}
in
$B_{r/4}(x)\times(t-U^{1-m}r^{2}/16,t]$. For simplicity, denote
\begin{equation*}
u_{*}=u(x,t),
\qquad
v=\frac{m}{m-1}u^{m-1},
\qquad
v_{*}=v(x,t).
\end{equation*}
Since $U\leqslant4u_{*}$, we have
\begin{equation}\label{equ;26}
v_{*}=\frac{m}{m-1}u_{*}^{m-1},
\qquad
\frac{U^{m-1}}{v_{*}}
\leqslant\frac{m-1}{m}4^{m-1}.
\end{equation}
Choose $\eta=\eta(m,n)>0$ small enough such that
\begin{equation}\label{equ;27}
\begin{aligned}
&\eta\leqslant\frac14,
\qquad
\eta^{2}\frac{m-1}{m}4^{m-1}\leqslant\frac1{16},\\
&4C_{H}
\left[
\eta\left(1+\sqrt{\frac{m-1}{m}4^{m-1}}\right)
\right]^{\alpha_{H}}
\leqslant
\min\left\{
\begin{aligned}
&\frac12,
\quad 1-\left(1-\frac1{8n}\right)^{1/(m-1)},\\
&\left(1+\frac1{8n}\right)^{1/(m-1)}-1
\end{aligned}
\right\}.
\end{aligned}
\end{equation}
Then $B_{\eta r}(x)\subset B_{r/4}(x)$ and
\begin{equation*}
\frac{\eta^{2}r^{2}}{v_{*}}
\leqslant\frac1{16}U^{1-m}r^{2}.
\end{equation*}
Applying (\ref{equ;25}) in
\begin{equation*}
B_{\eta r}(x)\times
\left(t-\frac{\eta^{2}r^{2}}{v_{*}},t\right]
\end{equation*}
and using (\ref{equ;26})--(\ref{equ;27}), we obtain
\begin{align*}
|u(y,\tau)-u_{*}|
&\leqslant
C_{H}U
\left[
\eta\left(1+\sqrt{\frac{U^{m-1}}{v_{*}}}\right)
\right]^{\alpha_{H}}\\
&\leqslant
4C_{H}
\left[
\eta\left(1+\sqrt{\frac{m-1}{m}4^{m-1}}\right)
\right]^{\alpha_{H}}u_{*}.
\end{align*}
Consequently,
\begin{equation*}
\left(1-\frac1{8n}\right)^{1/(m-1)}
\leqslant\frac{u(y,\tau)}{u_{*}}
\leqslant
\left(1+\frac1{8n}\right)^{1/(m-1)},
\end{equation*}
or equivalently,
\begin{equation}\label{equ;28}
\left(1-\frac1{8n}\right)v_{*}
\leqslant v(y,\tau)
\leqslant
\left(1+\frac1{8n}\right)v_{*}
\end{equation}
throughout the same cylinder.

Now we introduce the linear operator
\begin{equation*}
L=\frac{\partial}{\partial t}-(m-1)v\Delta
\end{equation*}
and the auxiliary function
\begin{equation*}
F=\frac{|\nabla v|^{2}}{(B-v)^{2}},
\qquad
B=\left(1+\frac1{4n}\right)v_{*}.
\end{equation*}
By (\ref{equ;28}),
\begin{equation*}
\frac{v_{*}}{8n}\leqslant B-v\leqslant\frac{3v_{*}}{8n}.
\end{equation*}
A common computation yields
\begin{align*}
L(|\nabla v|^{2})
&=2(m-1)|\nabla v|^{2}\Delta v
-2(m-1)v|\nabla^{2}v|^{2}
+2\langle\nabla v,\nabla|\nabla v|^{2}\rangle\\
&\quad-2(m-1)v\mathrm{Ric}(\nabla v,\nabla v).
\end{align*}
Consequently,
\begin{align}
LF
&=\frac1{(B-v)^{2}}
\bigl[
2(m-1)|\nabla v|^{2}\Delta v
-2(m-1)v|\nabla^{2}v|^{2}
+2\langle\nabla v,\nabla|\nabla v|^{2}\rangle\notag\\*
&\qquad\qquad
-2(m-1)v\mathrm{Ric}(\nabla v,\nabla v)
\bigr]
+\frac{2|\nabla v|^{2}v_{t}}{(B-v)^{3}}\notag\\*
&\quad
-(m-1)v|\nabla v|^{2}
\left(
\frac{2\Delta v}{(B-v)^{3}}
+\frac{6|\nabla v|^{2}}{(B-v)^{4}}
\right)
-\frac{4(m-1)v}{(B-v)^{3}}
\langle\nabla|\nabla v|^{2},\nabla v\rangle\notag\\
&=2(m-1)F\Delta v
-\frac{2(m-1)v}{(B-v)^{2}}|\nabla^{2}v|^{2}
+2\left(1-\frac{2(m-1)v}{B-v}\right)
\langle\nabla v,\nabla F\rangle\notag\\*
&\quad
-\frac{2(m-1)v}{(B-v)^{2}}
\mathrm{Ric}(\nabla v,\nabla v)
+2(mv-B)F^{2}.\label{equ;29}
\end{align}
At a point where $|\nabla v|>0$, choose a local orthonormal frame
$\{e_{1},\ldots,e_{n}\}$ with
$e_{1}=\nabla v/|\nabla v|$.  Then $v_{1}=|\nabla v|$ and
$v_{j}=0$ for $j\neq1$.  Since
\begin{equation*}
\nabla F
=\frac{\nabla|\nabla v|^{2}}{(B-v)^{2}}
+\frac{2|\nabla v|^{2}\nabla v}{(B-v)^{3}},
\end{equation*}
we have
\begin{equation}\label{equ;30}
F_{1}=\frac{2|\nabla v|v_{11}}{(B-v)^{2}}
+\frac{2|\nabla v|^{3}}{(B-v)^{3}},
\qquad
F_{j}=\frac{2|\nabla v|v_{1j}}{(B-v)^{2}}
\quad(j\neq1).
\end{equation}
It follows that
\begin{equation}\label{equ;31}
\Delta v
=\frac{(B-v)^{2}}{2|\nabla v|}F_{1}
-\frac{|\nabla v|^{2}}{B-v}
+\sum_{i=2}^{n}v_{ii},
\end{equation}
and
\begin{equation}\label{equ;32}
\begin{aligned}
|\nabla^{2}v|^{2}
&=\frac{(B-v)^{4}}{4|\nabla v|^{2}}F_{1}^{2}
-(B-v)|\nabla v|F_{1}
+\frac{|\nabla v|^{4}}{(B-v)^{2}}\\
&\quad
+\frac{(B-v)^{4}}{2|\nabla v|^{2}}
\sum_{i=2}^{n}F_{i}^{2}
+\sum_{i,j=2}^{n}v_{ij}^{2}.
\end{aligned}
\end{equation}
Substituting (\ref{equ;31})--(\ref{equ;32}) into (\ref{equ;29}),
\begin{align*}
LF
&=\frac{(m-1)(B-v)^{2}}{|\nabla v|}F_{1}F
-\frac{2(m-1)|\nabla v|^{2}}{B-v}F
-\frac{(m-1)v(B-v)^{2}}{2|\nabla v|^{2}}F_{1}^{2}\\
&\quad
+\frac{2(m-1)v|\nabla v|}{B-v}F_{1}
-\frac{2(m-1)v|\nabla v|^{4}}{(B-v)^{4}}
+2\left(1-\frac{2(m-1)v}{B-v}\right)
\langle\nabla v,\nabla F\rangle\\
&\quad
-\frac{2(m-1)v}{(B-v)^{2}}
\mathrm{Ric}(\nabla v,\nabla v)
+2(mv-B)F^{2}
-\frac{(m-1)v(B-v)^{2}}{|\nabla v|^{2}}
\sum_{i=2}^{n}F_{i}^{2}\\
&\quad
+\left[
2(m-1)F\sum_{i=2}^{n}v_{ii}
-\frac{2(m-1)v}{(B-v)^{2}}
\sum_{i,j=2}^{n}v_{ij}^{2}
\right].
\end{align*}
The last bracket is bounded above by
\begin{align*}
&2(m-1)F\sum_{i=2}^{n}v_{ii}
-\frac{2(m-1)v}{(B-v)^{2}}
\sum_{i,j=2}^{n}v_{ij}^{2}\\
&\qquad\leqslant 2(m-1)\sum_{i=2}^{n}
\left(Fv_{ii}-\frac{v}{(B-v)^{2}}v_{ii}^{2}\right)\leqslant 
\frac{(m-1)(n-1)(B-v)^{2}}{2v}F^{2}.
\end{align*}
Dropping the remaining nonpositive square terms, we obtain
\begin{align*}
LF
&\leqslant
-\left[
2m(B-v)-\frac{(m-1)(n-1)(B-v)^{2}}{2v}
\right]F^{2}\\
&\quad
+\left[m+1-\frac{2(m-1)v}{B-v}\right]
\langle\nabla v,\nabla F\rangle
-\frac{2(m-1)v}{(B-v)^{2}}
\mathrm{Ric}(\nabla v,\nabla v).
\end{align*}
Moreover, (\ref{equ;28}) implies
\begin{equation*}
\frac{(m-1)(n-1)(B-v)^{2}}{2v}
\leqslant
\frac{3(m-1)(n-1)}{2(8n-1)}(B-v)
<m(B-v).
\end{equation*}
Therefore
\begin{equation*}
-LF
\geqslant
\frac{m}{8n}v_{*}F^{2}
+\left[
\frac{2(m-1)v}{B-v}-(m+1)
\right]\langle\nabla v,\nabla F\rangle
-2(m-1)kvF.
\end{equation*}

Choose a space-time cut-off function $\varphi$ such that
\begin{equation*}
\operatorname{Supp}\varphi
\subset
B_{\eta r}(x)\times
\left(t-\frac{\eta^{2}r^{2}}{v_{*}},t\right],
\end{equation*}
\begin{equation*}
\varphi\equiv1
\quad\text{on}\quad
B_{\eta r/2}(x)\times
\left(t-\frac{\eta^{2}r^{2}}{2v_{*}},t\right],
\qquad
0\leqslant\varphi\leqslant1,
\end{equation*}
and
\begin{equation*}
\frac{|\nabla\varphi|}{\varphi^{1/2}}
\leqslant\frac{C(m,n)}{r},
\qquad
|\varphi_{t}|\leqslant\frac{C(m,n)v_{*}}{r^{2}},
\qquad
\Delta\varphi
\geqslant-C(m,n)\left(k+\frac1{r^{2}}\right)
\geqslant-\frac{C(m,n)}{r^{2}}.
\end{equation*}
from the Laplace comparison theorem. We next calculate at a positive maximum point of $\varphi F$ and assume that this point is not in the cut locus, otherwise we apply the Calabi trick.  At this point, $\varphi\nabla F=-F\nabla\varphi$, and the preceding differential
inequality yields
\begin{align*}
0
&\geqslant-L(\varphi F)\\
&\geqslant \frac{m}{8n}v_{*}\varphi F^{2}
-\left[
\frac{2(m-1)v}{B-v}-(m+1)
\right]F\left<\nabla v, \nabla\varphi\right>
-2(m-1)kv\varphi F\\
&\quad -\varphi_{t}F+(m-1)vF\Delta \varphi+2(m-1)v\left<\nabla \varphi, \nabla F\right>\\
&\geqslant
\frac{m}{8n}v_{*}\varphi F^{2}
-\left[
\frac{2(m-1)v}{B-v}+(m+1)
\right]F|\nabla v||\nabla\varphi|
-2(m-1)kv\varphi F\\
&\quad
-\frac{C(m,n)}{r^{2}}v_{*}F
-\frac{C(m,n)}{r^{2}}vF.
\end{align*}
The second term satisfies
\begin{align*}
&-\left[
\frac{2(m-1)v}{B-v}+(m+1)
\right]F|\nabla v||\nabla\varphi|\\
&\quad\geqslant
-\bigl[2(m-1)v+(m+1)(B-v)\bigr]
F^{3/2}\frac{C(m,n)}{r}\varphi^{1/2}\\
&\quad\geqslant
-\frac{C(m,n)}{r}\sqrt{\varphi F}\,v_{*}F.
\end{align*}
Using $kr^{2}\leqslant1/64$, $v\leqslant(1+1/(8n))v_{*}$, and
$0\leqslant\varphi\leqslant1$, we conclude that
\begin{equation*}
0\geqslant
\frac{m}{8n}v_{*}\varphi F^{2}
-\frac{C(m,n)}{r}\sqrt{\varphi F}\,v_{*}F
-\frac{C(m,n)}{r^{2}}v_{*}F.
\end{equation*}
Dividing by $v_{*}F$ gives
\begin{equation*}
\frac{m}{8n}\varphi F
-\frac{C(m,n)}{r}\sqrt{\varphi F}
-\frac{C(m,n)}{r^{2}}\leqslant0,
\end{equation*}
and hence
\begin{equation*}
\varphi F\leqslant\frac{C(m,n)}{r^{2}}.
\end{equation*}
In particular, at $(x,t)$,
\begin{equation*}
\varphi F
=\frac{|\nabla v|^{2}}{(B-v_{*})^{2}}
\geqslant\frac{64n^{2}|\nabla v|^{2}}{9v_{*}^{2}}.
\end{equation*}
Thus $|\nabla v|(x,t)\leqslant C(m,n)v_{*}/r$, which is equivalent to
(\ref{equ;23}).
\end{proof}

For the FDE case, we prove the following result:
\begin{lemma}\label{lemma;local-gradient2}
Let $u\in C^{2,1}_{\mathrm{loc}}$ be a positive  solution of (\ref{equ;1}) for $0<m<1$, and let
$(x,t)\in B_{R}(x_{0})\times(t_{0}-T/2,t_{0}]$.  Suppose that $\mathrm{Ric}_{M}\geqslant-k$, $\sqrt{k}r\leqslant1/8$, and that for some $U>0$,
\begin{equation}\label{equ;49}
\frac{U}{4}\leqslant u\leqslant U
\quad\text{in}\quad
B_{r/2}(x)\times
\left(t-\frac14U^{1-m}r^{2},t\right].
\end{equation}
Then there exists a constant $C=C(m,n)$ such that
\begin{equation}\label{equ;50}
|\nabla u|(x,t)\leqslant C(m,n)\frac{u(x,t)}{r}.
\end{equation}
\end{lemma}

\begin{proof}
The linear H\"older step is the same as in the proof of Lemma \ref{lemma;local-gradient}, except for the ellipticity interval. With $a(y,\tau)=mu^{m-1}(y,\tau)$, the time change $\theta=U^{m-1}(\tau-t)$ gives
\begin{equation*}
 \widehat u_\theta=\operatorname{div}(\widehat a (y,\theta)\nabla\widehat u),
 \qquad m\leqslant\widehat a\leqslant m4^{1-m}.
\end{equation*}
By the local H\"older continuity, there exist $C_H=C_H(m,n)$ and $\alpha_H=\alpha_H(m,n)\in(0,1)$ such that
\begin{equation}\label{equ;51}
 |u(y,\tau)-u(x,t)|\leqslant C_HU
 \left(\frac{d(x,y)+\sqrt{U^{m-1}(t-\tau)}}r\right)^{\alpha_H}
\end{equation}
in $B_{r/4}(x)\times(t-U^{1-m}r^2/16,t]$. For simplicity, denote 
\begin{equation}\label{equ;52}
 u_*=u(x,t),\qquad
 v=\frac m{1-m}u^{m-1},\qquad v_*=v(x,t).
\end{equation}
In this case, we have
\begin{equation}\label{equ;53}
 v_t=(1-m)v\Delta v-|\nabla v|^2,
 \qquad
 \frac{U^{m-1}}{v_*}
 =\frac{1-m}{m}\left(\frac U{u_*}\right)^{m-1}
 \leqslant\frac{1-m}{m}.
\end{equation}
Choose
\begin{equation}\label{equ;54}
 \varepsilon=\min\left\{\frac18,\frac{m}{8(1-m)n}\right\}>0,
\end{equation}
and then $\eta=\eta(m,n)>0$  small enough such that
\begin{equation}\label{equ;55}
\begin{aligned}
 &\eta\leqslant\frac18,\qquad \eta^2\frac{1-m}{m}\leqslant\frac1{16},\\
 &4C_H\left[\eta\left(1+\sqrt{\frac{1-m}{m}}\right)\right]^{\alpha_H}
 \leqslant\min\left\{\frac12,
       1-(1+\varepsilon)^{1/(m-1)},
       (1-\varepsilon)^{1/(m-1)}-1\right\}.
\end{aligned}
\end{equation}
The minimum on the right is strictly positive. Since $U\leqslant4u_*$, \eqref{equ;51}--\eqref{equ;55} imply
\begin{equation*}
 (1+\varepsilon)^{1/(m-1)}
 \leqslant\frac{u(y,\tau)}{u_*}
 \leqslant(1-\varepsilon)^{1/(m-1)}
\end{equation*}
on $B_{\eta r}(x)\times(t-\eta^2r^2/v_*,t]$. Raising to the negative power $m-1$ reverses the inequalities, and gives 
\begin{equation}\label{equ;56}
 (1-\varepsilon)v_*\leqslant v(y,\tau)\leqslant(1+\varepsilon)v_*
\end{equation}
throughout the same cylinder.

 Introduce the linear operator $L=\partial_t-(1-m)v\Delta $ and the auxiliary function 
 \begin{equation*}
    F=\frac{|\nabla v|^2}{(v-B)^2},\qquad  B=(1-2\varepsilon)v_*.
 \end{equation*}
 By \eqref{equ;56}, 
 \begin{equation}\label{equ;57}
 \varepsilon v_*\leqslant v-B\leqslant3\varepsilon v_*.
\end{equation}
A common computation yields
\begin{align*}
 L(|\nabla v|^2)={}&2(1-m)|\nabla v|^2\Delta v
       -2(1-m)v|\nabla^2v|^2\\
 &-2\langle\nabla v,\nabla|\nabla v|^2\rangle
       -2(1-m)v\Ric(\nabla v,\nabla v).
\end{align*}
The corresponding calculation to that in the PME proof implies
\begin{equation}\label{equ;58}
\begin{aligned}
 LF={}&2(1-m)F\Delta v
       -\frac{2(1-m)v}{(v-B)^2}|\nabla^2v|^2\\
 &+2\left(-1+\frac{2(1-m)v}{v-B}\right)
       \langle\nabla v,\nabla F\rangle
       -\frac{2(1-m)v}{(v-B)^2}\Ric(\nabla v,\nabla v)\\
 &+2(B-mv)F^2.
\end{aligned}
\end{equation}
At a point where $|\nabla v|>0$, choose a local orthonormal frame
$\{e_{1},\ldots,e_{n}\}$ with
$e_{1}=\nabla v/|\nabla v|$.  Then $v_{1}=|\nabla v|$ and
$v_{j}=0$ for $j\neq1$.  Then we have
\begin{equation}\label{equ;59}
F_{1}=\frac{2|\nabla v|v_{11}}{(v-B)^{2}}
-\frac{2|\nabla v|^{3}}{(v-B)^{3}},
\qquad
F_{j}=\frac{2|\nabla v|v_{1j}}{(v-B)^{2}}
\quad(j\neq1).
\end{equation}
It follows that
\begin{equation}\label{equ;60}
\Delta v
=\frac{(v-B)^{2}}{2|\nabla v|}F_{1}
+\frac{|\nabla v|^{2}}{v-B}
+\sum_{i=2}^{n}v_{ii},
\end{equation}
and
\begin{equation}\label{equ;61}
\begin{aligned}
|\nabla^{2}v|^{2}
&=\frac{(v-B)^{4}}{4|\nabla v|^{2}}F_{1}^{2}
+(v-B)|\nabla v|F_{1}
+\frac{|\nabla v|^{4}}{(v-B)^{2}}\\
&\quad
+\frac{(v-B)^{4}}{2|\nabla v|^{2}}
\sum_{i=2}^{n}F_{i}^{2}
+\sum_{i,j=2}^{n}v_{ij}^{2}.
\end{aligned}
\end{equation}
Substituting \eqref{equ;60}-\eqref{equ;61} into \eqref{equ;58} and dropping the remaining nonpositive square terms, we obtain
\begin{equation}\label{equ;62}
\begin{aligned}
 LF\leqslant{}&-\left[2m(v-B)-\frac{(1-m)(n-1)(v-B)^2}{2v}\right]F^2\\
 &+\left[-(1+m)+\frac{2(1-m)v}{v-B}\right]
       \langle\nabla v,\nabla F\rangle
       -\frac{2(1-m)v}{(v-B)^2}\Ric(\nabla v,\nabla v).
\end{aligned}
\end{equation}
Here we have used
\begin{equation*}
 2(1-m)F\sum_{i=2}^nv_{ii}
       -\frac{2(1-m)v}{(v-B)^2}\sum_{i,j=2}^nv_{ij}^2\leqslant
       \frac{(1-m)(n-1)(v-B)^2}{2v}F^2.
\end{equation*}
The choice of $\varepsilon$ (see \eqref{equ;54}) ensures 
\begin{equation*}
 \frac{(1-m)(n-1)(v-B)}{2v}
 \leqslant\frac{3(1-m)(n-1)\varepsilon}{2(1-\varepsilon)}
 \leqslant\frac{3m(n-1)}{14n}<m.
\end{equation*}
Consequently,
\begin{equation}\label{equ;63}
 2m(v-B)-\frac{(1-m)(n-1)(v-B)^2}{2v}
 \geqslant m(v-B)\geqslant m\varepsilon v_*>0.
\end{equation}
Therefore
\begin{equation}\label{equ;64}
\begin{aligned}
 -LF\geqslant{}&m\varepsilon v_*F^2
 +\left[1+m-\frac{2(1-m)v}{v-B}\right]
        \langle\nabla v,\nabla F\rangle
 -2(1-m)kvF.
\end{aligned}
\end{equation}
Then we follow the remaining steps in the proof of Lemma \ref{lemma;local-gradient} and thus conclude \eqref{equ;50}.
\end{proof}


\section{Proof of Theorem \ref{thm;main} and Theorem \ref{thm;main2}}

\begin{proof}[Proof of Theorem \ref{thm;main}]  Fix any
$(x,t)\in B_{R}(x_{0})\times(t_{0}-T/2,t_{0}]$
and write $u_{*}=u(x,t)>0$.  Define
\begin{equation}\label{equ;68}
\rho=\min\left\{
\frac{R}{4},
\frac14\sqrt{A^{m-1}T},
\frac{1}{8\sqrt{k}}
\right\},
\end{equation}
where the last entry is understood as $+\infty$ when $k=0$.  Then
\begin{equation}\label{equ;69}
B_{2\rho}(x)\times(t-4A^{1-m}\rho^{2},t]
\subset
B_{2R}(x_{0})\times(t_{0}-T,t_{0}],
\qquad
\sqrt{k}\rho\leqslant\frac18.
\end{equation}

Let $\nu$ and $\sigma$ be the constants in Lemma
\ref{lemma;alternative}, and choose
\begin{equation}\label{equ;70}
\beta
=\max\left\{
m,
\frac{m-1}{2}
+\frac{\log(4/\sqrt{\nu})}{\log(1/\sigma)}
\right\}.
\end{equation}
Then
\begin{equation*}
\sigma^{\beta-(m-1)/2}
\leqslant
\sigma^{\log(4/\sqrt{\nu})/\log(1/\sigma)}
=\frac{\sqrt{\nu}}{4}.
\end{equation*}
For $j=0,1,2,\ldots$, define
\begin{equation*}
U_{j}=\sigma^{j}A,
\qquad
\rho_{j}=\sigma^{\beta j}\rho,
\qquad
Q_{j}=B_{2\rho_{j}}(x)
\times(t-4U_{j}^{1-m}\rho_{j}^{2},t].
\end{equation*}
By (\ref{equ;69}), $0<u\leqslant U_{0}=A$ on $Q_{0}$.  Suppose that
$0<u\leqslant U_{j}$ on $Q_{j}$.  Applying Lemma
\ref{lemma;alternative} with $U=U_{j}$ and $r=\rho_{j}$, one of its two
alternatives occurs.  If alternative (i) occurs, then by the continuity of $u$ 
\begin{equation*}
\frac{U_{j}}{4}\leqslant u\leqslant U_{j}
\quad\text{in}\quad
B_{\rho_{j}/2}(x)\times
\left(t-\frac14U_{j}^{1-m}\rho_{j}^{2},t\right],
\end{equation*}
and Lemma \ref{lemma;local-gradient} yields
\begin{equation*}
|\nabla u|(x,t)
\leqslant C(m,n)\frac{u(x,t)}{\rho_{j}}.
\end{equation*}
If alternative (ii) occurs, then
\begin{equation*}
0<u\leqslant\sigma U_{j}=U_{j+1}
\quad\text{in}\quad
B_{\rho_{j}/2}(x)\times
\left(t-\frac{\nu}{4}U_{j}^{1-m}\rho_{j}^{2},t\right].
\end{equation*}
The choice of $\beta$ implies
\begin{equation*}
2\rho_{j+1}=2\sigma^{\beta}\rho_{j}\leqslant \frac{\sqrt{\nu}}{2}\sigma^{(m-1)/2}\rho_{j}\leqslant\frac{\rho_{j}}{2},
\end{equation*}
and
\begin{equation*}
4U_{j+1}^{1-m}\rho_{j+1}^{2}
=4\sigma^{2\beta+1-m}U_{j}^{1-m}\rho_{j}^{2}
\leqslant\frac{\nu}{4}U_{j}^{1-m}\rho_{j}^{2}.
\end{equation*}
Therefore $Q_{j+1}$ is contained in the preceding cylinder, and the
iteration continues.

Alternative (ii) cannot occur for every $j$, since otherwise
\begin{equation*}
0<u_{*}\leqslant U_{j}=\sigma^{j}A\longrightarrow0.
\end{equation*}
Hence alternative (i) occurs for some finite $J$.  For this $J$,
$u_{*}\leqslant U_{J}=\sigma^{J}A$, and consequently
\begin{equation*}
\rho_{J}=\sigma^{\beta J}\rho
\geqslant\rho\left(\frac{u_{*}}{A}\right)^{\beta}.
\end{equation*}
It follows that
\begin{align*}
|\nabla u^{\beta}|(x,t)
&=\beta u_{*}^{\beta-1}|\nabla u|(x,t)\\
&\leqslant C(m,n)\frac{u_{*}^{\beta}}{\rho_{J}}
\leqslant C(m,n)\frac{A^{\beta}}{\rho}.
\end{align*}
Finally,
\begin{equation*}
\frac1{\rho}
=\max\left\{
\frac4R,
\frac{4A^{-(m-1)/2}}{\sqrt T},
8\sqrt{k}
\right\}
\leqslant
\frac4R+\frac{4A^{-(m-1)/2}}{\sqrt T}+8\sqrt{k}.
\end{equation*}
Since $(x,t)$ was arbitrary, (\ref{equ;2}) follows.
\end{proof}

\begin{proof}[Proof of Theorem \ref{thm;main2}] Fix any
$(x,t)\in B_{R}(x_{0})\times(t_{0}-T/2,t_{0}]$
and write $u_{*}=u(x,t)>0$.  Define, exactly as in the proof of Theorem \ref{thm;main},
\begin{equation}\label{equ;65}
\rho=\min\left\{
\frac{R}{4},
\frac14\sqrt{A^{m-1}T},
\frac{1}{8\sqrt{k}}
\right\},
\end{equation}
where the last entry is understood as $+\infty$ when $k=0$.  Then
\begin{equation}\label{equ;66}
B_{2\rho}(x)\times(t-4A^{1-m}\rho^{2},t]
\subset
B_{2R}(x_{0})\times(t_{0}-T,t_{0}],
\qquad
\sqrt{k}\rho\leqslant\frac18.
\end{equation}

Let $\nu$ and $\sigma$ be the constants in Lemma \ref{lemma;alternative2}. The same formula for the exponent is sufficient:
\begin{equation}\label{equ;67}
 \beta=\max\left\{m,\frac{m-1}{2}
            +\frac{\log(4/\sqrt\nu)}{\log(1/\sigma)}\right\}.
\end{equation}
The remaining steps of the proof of Theorem \ref{thm;main} can then be repeated verbatim, and \eqref{equ;71} holds in $B_{R}(x_{0})\times(t_{0}-T/2,t_{0}]$. 
\end{proof}

\medskip

\vspace{.1in}
\textbf{Acknowledgements} The first author is supported by NSFC No. 12671074, 12531002, 12271039.

\vspace{.1in}
\textbf{Conflict of Interest} The authors have no conflict of interest to declare.

\vspace{.1in}
\textbf{Data availability} The authors declare no datasets were generated or analysed during the current study.

\vspace{.1in}
\textbf{AI assistance statement}  During the preparation of this work the authors used ChatGPT 5.6 in order to conduct literature searches and summaries, polish the language of the manuscript, review the proofs and assist with calculations for Lemma \ref{lemma;local-gradient} based on explicit computational paths defined by the authors. After using this tool, the authors reviewed and edited the content as needed and take full responsibility for the content of the publication.

\end{document}